\documentclass[11pt]{amsart}

\usepackage{amssymb, amsthm, amsmath, gensymb}
\usepackage{graphicx, comment}
\usepackage[small]{caption}
\usepackage{subcaption}
\usepackage{epsfig}
\usepackage{tikz, pgfplots, float}

\newcommand{\later}[1]{}
\newcommand{\old}[1]{}

\usepackage{amsfonts}

\usepackage[utf8]{inputenc}
\usepackage{fullpage}
\usepackage{framed}

\usepackage{enumerate}
\usepackage{url}
\usepackage[breaklinks]{hyperref}
\hypersetup{
	colorlinks = true, 
	urlcolor = cyan, 
	linkcolor = teal, 
	citecolor = cyan 
}

\newtheorem{theorem}{Theorem}[section]
\newtheorem{lemma}[theorem]{Lemma}

\newtheorem{observation}[theorem]{Observation}

\newtheorem{proposition}[theorem]{Proposition}
\newtheorem{corollary}[theorem]{Corollary}
\newtheorem{problem}[theorem]{Problem}

\newtheorem{conjecture}[theorem]{Conjecture}

\newcommand{\ie}{{i.e.}}
\newcommand{\eg}{{e.g.}}

\newcommand{\eps}{\varepsilon}

\newcommand{\dist}{{\rm dist}}

\newcommand{\x}{{\rm cr}}

\newcommand{\NN}{\mathbb{N}} 
\newcommand{\RR}{\mathbb{R}} 

\def\R{\mathcal R}

\newcommand{\I}{\mathrm{i}}

\usepackage{tikz}
\tikzstyle{vtx} = [circle, fill, inner sep=0.7]

\title{On multiplicities of interpoint distances}

\author{Felix Christian Clemen}
\address{Felix Christian Clemen \newline University of Victoria, Canada}
\email{fclemen@uvic.ca}
\thanks{The first author is partially supported by a PIMS Postdoctoral Fellowship.}

\author{Adrian Dumitrescu}
\address{Adrian Dumitrescu \newline Algoresearch L.L.C., Milwaukee, WI, USA}
\email{ad.dumitrescu@algoresearch.org}

\author{Dingyuan Liu}
\address{Dingyuan Liu \newline Karlsruhe Institute of Technology, Germany}
\email{liu@mathe.berlin}

\date{}

\begin{document}

\begin{abstract}
Given a set $X\subseteq\RR^2$ of $n$ points and a distance $d>0$, the multiplicity of $d$ is the
number of times the distance $d$ appears between points in $X$. Let $a_1(X) \geq a_2(X) \geq \cdots \geq a_m(X)$ 
denote the multiplicities of the $m$ distances determined by $X$ and let
$a(X)=\left(a_1(X),\dots,a_m(X)\right)$. In this paper, we study several
questions from Erd\H{o}s's time regarding distance multiplicities. Among other results,
we show that: 
\begin{enumerate}
\item If $X$ is convex or ``not too convex'', then there exists a distance other than the diameter
  that has multiplicity at most $n$.  
\item There exists a set $X \subseteq \RR^2$ of $n$ points, such that many distances occur with high
  multiplicity. In particular, at least $n^{\Omega(1/\log\log{n})}$ distances 
  have superlinear multiplicity in $n$.
\item For any (not necessarily fixed) integer $1\leq k\leq\log{n}$,
  there exists $X\subseteq\RR^2$ of $n$ points, such that the difference between
  the $k^{\text{th}}$ and $(k+1)^{\text{th}}$ largest multiplicities is at least
  $\Omega(\frac{n\log{n}}{k})$. 
  Moreover, the distances in $X$ with the largest $k$ multiplicities can be prescribed.
\item For every $n\in\NN$, there exists $X\subseteq\RR^2$ of $n$ points, not all
  collinear or cocircular, such that $a(X)= (n-1,n-2,\ldots,1)$. There also
  exists $Y\subseteq\RR^2$ of $n$ points with pairwise distinct distance
  multiplicities and $a(Y) \neq (n-1,n-2,\ldots,1)$. 
\end{enumerate}

\medskip \noindent
\textbf{\small Keywords}: distance multiplicity, convex layer decomposition, integer grid.

\end{abstract}

\maketitle

\section{Introduction} \label{sec:intro}

Let $\dist(x,y)$ denote the Euclidean distance between points $x$ and $y$ in the plane.
Given a finite planar point set $X=\{x_1,\ldots,x_n\}$, let $d_1,\ldots,d_m$ denote
the distinct distances between points in $X$, where $m=m(X) \leq {n \choose 2}$. 
The \emph{multiplicity} of $d_k$ in $X$ is defined as
\[ a_k(X) = \lvert\{ (i,j) \colon 1 \leq i < j \leq n,\,\dist(x_i,x_j)=d_k\}\rvert. \]
We arrange the $m$ multiplicities as $a_1(X) \geq a_2(X) \geq \cdots \geq a_m(X)$,
irrespective to relative values of the $d_k$, 
and let $a(X)=\left(a_1(X),\dots,a_m(X)\right)$. In this paper, we revisit
several questions from the time of Erd\H{o}s regarding distance multiplicities: 

\begin{enumerate} \itemsep 3pt

\item \label{q1} Is it possible that all distances except the diameter have multiplicity
  larger than $n$? See~\cite{Er84} and~\cite[Conjecture~4]{EF95a}.

\item \label{q2} Can it happen that there are many distances of multiplicity at least $cn$,
  where $c>1$ is a constant, or even superlinear in $n$? See~\cite{EP90} and~\cite[Problem 11]{Er97}.

\item \label{q3} Estimate $\max_{X\subseteq\mathbb{R}^2,\,|X|=n} (a_1(X) -a_2(X))$, and more generally,
  \[\max_{X\subseteq\mathbb{R}^2,\,|X|=n} (a_k(X) -a_{k+1}(X))\] as well as possible.
  See~\cite[Section~3]{Er83}.

\item \label{q4} For sufficiently large $n\in\NN$, is it true that
  $a(X)=(n-1,n-2,\dots,1)$ if and only if $X$ consists of equidistant points on
  a line or on a circle? See~\cite[p.~135]{Er84}. 

\end{enumerate}

We answer Questions~\eqref{q2} and~\eqref{q4}, and give partial answers to the other two.

\subsection{Another distance with multiplicity at most $n$ besides the diameter}

The \emph{diameter} of $X$, denoted $\Delta=\Delta(X)$, is the maximum distance
between points in $X$. Further, denote by $\Delta_2=\Delta_2(X)$ and $\delta=\delta(X)$
the second largest and the smallest distances in $X$, respectively,
and by $\mu(X,d)$ the multiplicity of the distance $d$ in $X$. 

Hopf and Pannwitz~\cite{HP34} proved that the
multiplicity of the diameter among any $n$ points in the plane is at
most $n$. Erd\H{o}s~\cite{Er84} further conjectured that for any
$n$-element point set $X\subseteq\mathbb{R}^2$, there must be a second
distance besides the diameter that has multiplicity at most $n$. 

\begin{conjecture}[{Erd\H{o}s~\cite{Er84}, see also~\cite[Conjecture~4]{EF95a}}]
  \label{conj:second}
Let $n \geq5$. For any $X\subseteq\mathbb{R}^2$ with $|X|=n$, it is not
possible that every distance except the diameter occurs more than $n$ times. 
\end{conjecture}

The condition $n \geq 5$ is necessary, since for $n=4$ we can glue two equilateral triangles
of the same side length together as a rhombus and this gives a counterexample. Erd\H{o}s and
Fishburn~\cite{EF95a} proved the Conjecture~for $n=5,6$ and the case of $n \geq 7$ is still open.   
Here we confirm Conjecture~\ref{conj:second} in two special cases.
A point set $X\subseteq\mathbb{R}^2$ is said to be \emph{convex}, or
in \emph{convex position} if no point lies inside the convex hull of other points. 

\begin{theorem} \label{thm:second}
Let $n\geq5$. For any convex point set $X\subseteq\mathbb{R}^2$ with
$|X|=n$, it cannot happen that all distances except the diameter occur
more than $n$ times. 
\end{theorem}

Given a point set $X\subseteq\RR^2$, let $L_1=L_1(X)$ be the set of vertices of
the convex hull of $X$, called the \emph{first (outer) convex layer} of
$X$. Similarly, the \emph{second convex layer} $L_2=L_2(X)$ of $X$ is the set of
vertices of the convex hull of $X\setminus L_1$. Note that $X$ is convex if and only if $L_2$ is
empty. It follows from definition that $X$ is convex if and only if 
$|L_1|=|X|$. Next we confirm Conjecture~\ref{conj:second} for ``not too convex''
point sets, namely for point sets whose first and second convex layers are not
too large. 

\begin{theorem} \label{thm:dense}
Let $X\subseteq\RR^2$ be a set of $n \geq 2$ points. If
\[\min\left\{\frac{3}{2}(|L_1|+|L_2|),\,\frac{4}{3}|L_1|+2|L_2|,\,2|L_1|+|L_2|\right\} \leq n,\]
  then the second largest distance in $X$ can occur at most $n$ times.
\end{theorem}

Theorem~\ref{thm:dense}
directly implies that, if the ratio $\Delta(X)/\delta(X)$ of a set $X\subseteq\RR^2$
is small enough, then the second largest distance in $X$ occurs at most $n$ times. 

\begin{corollary}
\label{cor:dense}
If $X\subseteq\RR^2$ is a set of $n\in\NN$ points with $\Delta(X)\leq\frac{n}{3\pi}\delta(X)$,
then $\mu(X,\Delta_2)\leq n$.
\end{corollary}
\begin{proof}
Assume without loss of generality that $\delta(X)=1$, namely,
$\Delta(X)\leq\frac{n}{3\pi}$. Then, $|L_1|$ and $|L_2|$ are upper bounded by
the perimeter of the convex polygons formed by $L_1$ and $L_2$,
respectively. Since the perimeter of a convex polygon is at most $\pi$ times its
diameter, see~\cite[p.~76]{YB61}, we have $|L_1|+|L_2|\leq2n/3$ and thus
$\mu(X,\Delta_2)\leq n$ by Theorem~\ref{thm:dense}. 
\end{proof}

Note that the multiplicity of the second largest distance can be larger than $n$
in some planar point sets, see~\cite{Ve87,Ve96}. Thus, to fully resolve
Conjecture~\ref{conj:second}, one could perhaps consider multiplicities of different distances
simultaneously and show that one of them must be at most $n$. For instance, one could consider
the smallest and the second largest distance; however, as we demonstrate below,
their multiplicities can both be large. 

\begin{proposition} \label{thm:cons}
Let $m,n\in\NN$ with $m\leq\lfloor n/2 \rfloor$. There exists a planar point set
$X$ with $|X|=n$, such that $\mu(X,\Delta_2)\geq3m$ and $\mu(X,\delta)\geq 3n-5m+o(m)$. 
\end{proposition}

By letting $m=\lfloor3n/8\rfloor$, we obtain that $\min\{\mu(X,\Delta_2),\mu(X,\delta)\}\geq9n/8+o(n)$.
This also motivates the following problem.
\begin{problem} 
Determine 
\[ \limsup_{n \rightarrow \infty} \sup_{X\subseteq \mathbb{R}^2,\,|X|=n}\frac{ \min\{\mu(X,\Delta_2),\mu(X,\delta)\}}{n}. \]
\end{problem}

\subsection{Point sets with many large distance multiplicities}

Erd\H{o}s and Pach~\cite{EP90}, see also~\cite[Problem 11]{Er97}, asked the following question:
Given a set $X\subseteq\mathbb{R}^2$ of $n$ points, can it happen that
there are $c_1 n$ distances with multiplicities at least $c_2n$, for
some constant $c_1,c_2>0$? Bhowmick~\cite{Bh24} recently answered
their question in the positive: There exist arbitrary large planar
point sets $X$, $|X|=n$, such that there are $\lfloor n/4 \rfloor$
distances which occur at least $n+1$ times. Bhowmick~\cite{Bh24} also
considered higher multiplicities, distances that occur at least $n+m$
times, $m \geq 1$. He showed that there are sets with at least
$\lfloor \frac{n}{2(m+1)} \rfloor$ 
distances that occur at least $n+m$ times. Observe that for $m$ linear in $n$, this
lower bound is only $\Omega(1)$. Here we give a substantial
improvement by showing that there exist $X\subseteq\mathbb{R}^2$ of
$n$ points, such that at least $n^{c/\log\log{n}}$ distances have
superlinear multiplicity in $n$. For comparison purposes, note that
$n^{c/\log \log{n}} = \Omega((\log{n})^\alpha)$, for any fixed $\alpha>0$.  

\begin{theorem} \label{thm:many}
There exists some constant $c>0$ such that for sufficiently large $n\in\mathbb{N}$,
at least $n^{c/\log \log{n}}$ distances occur at least  $n^{1 + c/\log \log{n}}$ times
in the $\sqrt{n} \times \sqrt{n}$ grid. 
\end{theorem}

As mentioned earlier, Bhowmick~\cite{Bh24} answered the question of Erd\H{o}s and Pach~\cite{EP90}
with constants $c_1=1/4$ and $c_2=1$. One may ask whether the constant $c_1=1/4$ resulting
from his construction is the best possible. We extend the above investigation for the range $c_2>1$.  
More precisely, we show that there exist $n$-element planar point sets with $m$ distances
so that $c_1 m$ distances occur at least $c_2 n$ times, for suitable constants $c_1>0$, $c_2>1$.
Proposition~\ref{thm:m/9} below gives three sample combinations; these combinations are not exhaustive. 

\begin{proposition} \label{thm:m/9}
For every $\eps>0$, there exists $n_0(\eps) \in \NN$ such that if $n \geq n_0(\eps)$,
then out of the $m = \Theta(n/\sqrt{\log{n}})$ distances presented in the $\sqrt{n} \times \sqrt{n}$ grid:
\begin{enumerate} \itemsep 2pt
\item[(i)] at least $(1-\eps) \, m/9$ distances occur at least  $16n/9$ times;
\item[(ii)] at least $(1-\eps) \, m/16$ distances occur at least  $9n/4$ times;
\item[(iii)] at least $(1-\eps) \, m/25$ distances occur at least  $64n/25$ times.
\end{enumerate}
\end{proposition}

\subsection{On the differences $a_k(X) - a_{k+1}(X)$}

Let $f(n)$ denote the maximum value of $a_1(X)$ over all $X\subseteq\mathbb{R}^2$ with $|X|=n$.
Erd\H{o}s~\cite[Section~3]{Er83} asked whether $f(n)-a_2(X)$ tends to infinity as $n\to\infty$.
However, this question is somewhat ambiguous since it leaves the ground set $X$
unspecified. Here we reformulate the question and ask for
\[ \max_{X\subseteq\mathbb{R}^2,\,|X|=n} (a_1(X) -a_2(X)). \]
We show that the maximum difference $a_1(X)-a_2(X)$ is at least $\Omega(n\log{n})$.
This is implied by the following more general result. 

\begin{theorem} \label{thm:diff}
Let $n\in\mathbb{N}$ be sufficiently large and $1\leq k \leq \log{n}$. There
exists a point set $X\subseteq\RR^2$ with $|X|=n$, such that
$a_k(X) - a_{k+1}(X) = \Omega\left(\frac{n}{k} \log{n}\right)$. Moreover, the
distances with the largest $k$ multiplicities can be prescribed.  
\end{theorem}

In particular, $a_k(X) - a_{k+1}(X)$ can be superlinear in $n$ for $k \to \infty$.

\begin{corollary}
$\max_{X\subseteq\mathbb{R}^2,\,|X|=n} (a_1(X) -a_2(X))=\Omega(n\log{n})$.
\end{corollary}

Theorem~\ref{thm:many} suggests that the following stronger lower bound might be true:

\begin{problem} \label{problem:a1-a2}
Does there exist a constant $c>0$ such that for sufficiently large $n\in\mathbb{N}$, we have
  \[ \max_{X\subseteq\mathbb{R}^2,\,|X|=n} (a_1(X) -a_2(X)) \geq n^{1 + c/\log \log{n}} ? \] 
\end{problem}

\subsection{Related work}

Recall that $f(n)$ denotes the largest possible value of $a_1(X)$ among
all subsets $X\subseteq\RR^2$ of $n$ points. Determining $f(n)$, also known as
the unit distance problem, is notoriously difficult. The current best upper
bound is $O(n^{4/3})$ established by Spencer, Szemer\'edi, and Trotter~\cite{SST84}.
A simple and elegant argument based on crossing numbers is due to Sz\'ekely~\cite{Sz97}. From the other direction,
it is conjectured by Erd\H{o}s that a $\sqrt{n} \times \sqrt{n}$ section of the integer lattice
gives the correct order magnitude, $n^{1 + c/\log \log{n}}$, and so the current best 
upper bound seems far off. See also a recent survey by Szemer\'edi~\cite{Sz16} for more on this topic.

Let $A(n)$ be the maximum value of $\sum a_k(X)^2$ over all $X\subseteq\RR^2$
with $|X|=n$. Erd\H{o}s~\cite{Er90} asked whether $A(n)=O(n^3 (\log{n})^\alpha)$
holds for some positive constant $\alpha>0$. This question received a complete
answer via the work of Guth and Katz~\cite{GK15} on the problem of distinct
distances. Specifically, the authors proved that the inequality holds with 
$\alpha=1$, \ie, $A(n)= O(n^3 \log{n})$, and is tight in the $\sqrt{n} \times \sqrt{n}$ integer grid.
Lefmann and Thiele~\cite{LT95} proved that the sharper inequality $A(n)=O(n^3)$ holds for convex point sets;
this bound is tight, \eg, for a regular $n$-gon.

The rest of the paper is organized as follows. We will prove
Theorems~\ref{thm:second},~\ref{thm:dense} and Proposition~\ref{thm:cons} in Section~\ref{sec:second}.
In Section~\ref{sec:many} we will prove Theorem~\ref{thm:many} and Proposition~\ref{thm:m/9}.
Section~\ref{sec:diff} is devoted to proving Theorem~\ref{thm:diff}.
Finally, in Section~\ref{sec:distinct} we give a simple answer to Question~\eqref{q4}.

\section{A second multiplicity at most $n$ in planar point sets} \label{sec:second}

Given a finite point set $X\subseteq\RR^2$, recall that $\delta(X)<\Delta_2(X)<\Delta(X)$
denote the smallest, the second largest, and the largest distances in $X$, respectively
(assuming they exist and are different).

\subsection{The convex case}

\begin{proof}[Proof of Theorem~\ref{thm:second}]
Let $X\subseteq\RR^2$ be an arbitrary convex set of $n$ points. Let
$R_n$ denote a regular $n$-gon and $R_n^-$ denote a regular $n$-gon
minus one vertex.  
A classical result of Altman~\cite{Alt63} states that
$X$ determines at least $\lfloor n/2 \rfloor$ distinct distances; and this bound is attained by $R_n$.
Moreover, Altman proved that if $n$ is odd and $X$ determines exactly $\lfloor n/2 \rfloor$ distances,
then $X = R_n$; in particular, $a(X)= (n,n,\ldots,n)$. See also~\cite{EF95a,Fi96}. 

The complementary result for even $n$ is due to Fishburn~\cite{Fi95}.
Suppose that $n \geq 6$ is even and $X$~determines exactly $\lfloor n/2 \rfloor$ distances. Then
\begin{enumerate} \itemsep 2pt
    \item[(i)] for $n=6$, there exist exactly two possibilities, $a(X)=(6,6,3)$, or $a(X)=(5,5,5)$;
    \item[(ii)] for $n \geq 8$, either $X =R_n$ or $X=R_{n+1}^-$.
\end{enumerate}
In particular, for the second case we have $a(X)= (n,n,\ldots,n,n/2)$ or $a(X)= (n-1,n-1,\ldots,n-1)$.

We can now finalize the proof. If $X$ determines strictly more than $\lfloor n/2 \rfloor$
distinct distances and all distances smaller than $\Delta$ occur more than $n$ times,
then the number of point pairs is at least
\[ \left\lfloor \frac{n}{2} \right\rfloor (n+1) + 1 > {n \choose 2}, \]
a contradiction. Otherwise, $X$ determines exactly $\lfloor n/2 \rfloor$ distinct distances,
and it is easy to check that all possible cases listed previously satisfy the requirements.
\end{proof}

\subsection{The ``not too convex'' case} \label{constructions}

\begin{proof}[Proof of Theorem~\ref{thm:dense}]
Let $L_1=L_1(X)$ and $L_2=L_2(X)$ be the first and second convex layers of $X$. It suffices to show that
\[\mu(X,\Delta_2)\leq\min\left\{\frac{3}{2}(|L_1|+|L_2|),\,\frac{4}{3}|L_1|+2|L_2|,\,2|L_1|+|L_2|\right\}.\]

We have the following observations:

\begin{itemize} \itemsep 2pt
\item[(1)]\cite[Proposition~1]{Ve96} Let $p,q\in X$. If $\dist(p,q)=\Delta$ then $\{p,q\}\subseteq L_1$.
  If $\dist(p,q)=\Delta_2$ then $\{p,q\}\cap L_1 \neq\emptyset$.
\item[(2)] If $\dist(p,q)=\Delta_2$ and $p \in L_1$, then $q \in L_1 \cup L_2$. Indeed,
  the points of $X$ in the exterior of the circle of radius $\Delta_2$ centered at $p$, if any, have distance $\Delta$
  from $p$ and by observation (1) are in $L_1$. Hence, $q$ is at least in the second convex layer of $X$. 
\end{itemize}
Combining observations (1) and (2) we have that $\mu(X,\Delta_2)=\mu(L_1\cup L_2,\Delta_2)$,
moreover, $\Delta_2$ is still the second largest distance in $L_1\cup
L_2$. Vesztergombi~\cite{Ve87} showed that the multiplicity of the second
largest distance among any $n$ points in the plane is at most $3n/2$. Namely, we have
\[ \mu(X,\Delta_2)\leq\frac{3}{2}(|L_1|+|L_2|). \]

We proceed to prove $\mu(X,\Delta_2)\leq\min\left\{\frac{4}{3}|L_1|+2|L_2|,\,2|L_1|+|L_2|\right\}$.
Let $G$ be a graph on $L_1\dot\cup L_2$, where $pq$ is an edge if and only if
$\dist(p,q)=\Delta_2$. Then it holds that $\mu(X,\Delta_2)=e(G)$. Iteratively
remove vertices of degree less than $2$ from $G$, and let $G'$ be the remaining
graph whose vertex set is $L_1'\dot\cup L_2'$ with $L_1'\subseteq L_1$ and
$L_2'\subseteq L_2$. Then we have
\[e(G)\leq |L_1\setminus L_1'| + |L_2\setminus L_2'| + e(G').\]
To bound $e(G')$, we record several observations about the graph $G'$ by Vesztergombi~\cite{Ve96}: 

\begin{itemize} \itemsep 2pt
\item[(3)] $L_2'$ is an independent set (follows from (1) and (2)).
\item[(4)]\cite[Proposition~3]{Ve96} Every $q\in L_2'$ has degree exactly $2$.
\item[(5)]\cite[Proposition~5]{Ve96} Every $p\in L_1'$ has at most $2$ neighbors in $L_2'$.
\item[(6)]\cite[Proposition~6]{Ve96} If $p\in L_1'$ has $3$ neighbors in $L_1'$, then it has at most $1$ neighbor in $L_2'$.
\item[(7)]\cite[Proposition~7]{Ve96} If $p\in L_1'$ has $4$ neighbors in $L_1'$, then it has no neighbor in $L_2'$.  
\item[(8)]\cite[Proposition~8]{Ve96} Every $p\in L_1'$ has at most $4$ neighbors in $L_1'$.
\end{itemize}
Let $e(L_1')$ denote the number of edges in $L_1'$, and let $e(L_1',L_2')$
denote the number of edges between $L_1'$ and $L_2'$. Due to observations (3)--(5),
\[e(G')=e(L_1')+e(L_1',L_2')= e(L_1')+\,2|L_2'|.\]
Vesztergombi~\cite{Ve85} showed that the multiplicity of the second largest
distance among any $n$ points in convex position in the plane is at most
$4n/3$. Since $L_1'$ is convex and $\Delta_2$ is either the largest or second
largest distance in $L_1'$,  
$$e(L_1')=\mu(L_1',\Delta_2)\leq\frac{4}{3}|L_1'|.$$
On the other hand, let $\deg(p)$, $\deg_1(p)$, and $\deg_2(p)$ denote the number
of neighbors of $p$ in $G'$, $L_1'$, and $L_2'$, respectively. For each $p\in L_1'$,
by observations (4)--(8) we have $\deg(p)=\deg_1(p)+\deg_2(p)\leq4$. Therefore, 
\begin{align*}
e(L_1')=\frac{1}{2}\sum_{p\in L_1'}\deg_1(p)\leq \frac{1}{2}\sum_{p\in
  L_1'}\left( 4-\deg_2(p)\right)&=2|L_1'|-\frac{1}{2}e(L_1',L_2')\\ 
&=2|L_1'|-2|L_2'|.
\end{align*}
We conclude
\begin{align*}
e(G)&\leq|L_1\setminus L_1'|+|L_2\setminus L_2'|+e(G')=|L_1\setminus L_1'|+|L_2\setminus L_2'|+e(L_1')+e(L_1',L_2')\\
&\leq |L_1\setminus L_1'|+|L_2\setminus L_2'|+\min\left\{\frac{4}{3}|L_1'|,\,2|L_1'|-|L_2'|\right\}+2|L_2'| \\
&= |L_1\setminus L_1'|+|L_2\setminus L_2'|+\min\left\{\frac{4}{3}|L_1'|+2|L_2'|,\,2|L_1'|+|L_2'|\right\}\\
&\leq \min\left\{\frac{4}{3}|L_1|+2|L_2|,\,2|L_1|+|L_2|\right\}. \qedhere
\end{align*}
\end{proof}

\subsection{The case where the multiplicities of $\Delta_2$ and $\delta$ are both large}

\begin{proof}[Proof of Proposition~\ref{thm:cons}]
  Our construction is inspired by that of Vesztergombi~\cite{Ve87,Ve96}, see also~\cite[Chap.~5.8]{BMP05}.
  Let $m_1=m_2=m$ and $m_3=n-2m$. The construction comprises three groups,
  each containing $m_1,m_2$ and $m_3$ points, respectively. Note that $n=m_1+m_2+m_3$.   

\smallskip
\textbf{Group I:} Place the first $m_1$ points $v_1,\ldots, v_{m_1}$ as the vertices of a regular $m_1$-gon inscribed
in a circle $C$ of radius $n$. Let $\Delta$ and $\Delta_2$ be the largest and second largest distances in this $m_1$-gon.
Note that $\Delta=O(n)$ and the number of occurrences of $\Delta_2$ in Group I is exactly $m_1$.

\smallskip
\textbf{Group II:} The next $m_2$ points $u_1,\ldots, u_{m_2}$ are positioned inside the circle $C$ such that
$\dist(v_i, u_i) = \dist(u_i, v_{i+1}) = \Delta_2$, where the indices are modulo $m_2$. Notably,
the $u_i$'s lie on a circle. Let $\delta=\Theta(1)$ be the smallest distance in the construction so far,
representing the distance between consecutive $u_i$'s. This distance appears $m_2$ times in Group II.

\smallskip
\textbf{Group III:} The final $m_3$ points form an equilateral triangular lattice with mesh width $\delta$ contained 
in a disk of radius $\Theta(\sqrt{n})$. The lattice is centered at the origin, coinciding with the center
of the circle $C$. Note that the distance $\delta$ appears $3m_3+o(n)$ times inside Group III. 

\smallskip
Let $X$ denote the final construction, see Figure~\ref{fig:construction1} for an illustration.
Since $L_1(X)$ and $L_2(X)$ correspond to the points in Group I and Group II, respectively,
we have $|L_1|=|L_2|=m$. Overall, the distance $\delta$ appears $3m_3+m_2+o(n)=3n-5m+o(n)$ times,
and the distance $\Delta_2$ appears $m_1+2m_2=3m$ times.
We complete the proof by noting that $\delta(X)=\delta$ and $\Delta_2(X)=\Delta_2$. 
\end{proof}

\begin{figure}[htbp]
  
    \centering
    \resizebox{0.75\textwidth}{!}{ %
    \begin{tikzpicture}
        \def\n{7}  
        \def\R{2}  
        \def\r{1.0}  
        \def\squareSize{0.2} 

        \foreach \i in {1,...,\n} {
            \pgfmathsetmacro{\angle}{360*\i/\n}
            \coordinate (v\i) at (\angle:\R);
        }

        \foreach \i in {1,...,\n} {
            \pgfmathtruncatemacro{\next}{mod(\i,\n)+1}  
            \draw[thick] (v\i) -- (v\next);
        }

        \foreach \i in {1,...,\n} {
            \pgfmathsetmacro{\angle}{360*\i/\n}  
            \fill (v\i) circle (2pt);
            \node at (\angle:\R + 0.4) {$v_{\i}$};  
        }

        \foreach \i in {1,...,\n} {
            \pgfmathsetmacro{\angle}{360*\i/\n}
            \coordinate (u\i) at (\angle:\r);
            \fill (u\i) circle (2pt);
        }

        \foreach \i in {1,...,\n} {
            \pgfmathtruncatemacro{\newIndex}{mod(\i+2,\n)+1}  
            \node at (u\i) [shift={(0.2,-0.2)}] {$u_{\newIndex}$};  
        }

        \foreach \i in {1,...,\n} {
            \pgfmathtruncatemacro{\nextTwo}{mod(\i+1,\n)+1}
            \pgfmathtruncatemacro{\nextThree}{mod(\i+1+1+1,\n)+1}
            
            \draw[dotted, thick] (v\i) -- (v\nextTwo);

            \draw[dotted, thick] (u\i) -- (v\nextThree);

            \pgfmathtruncatemacro{\nextV}{mod(\i+1+1,\n)+1}  
            \draw[dotted, thick] (u\i) -- (v\nextV);
        }

        \draw[thick] 
            (-\squareSize, -\squareSize) -- (\squareSize, -\squareSize) -- 
            (\squareSize, \squareSize) -- (-\squareSize, \squareSize) -- cycle;

        \def\s{3.5} 
        \def\cx{6} 
        \def\cy{0} 

        \draw[thick] 
    (\cx-\s/2, \cy-\s/2) -- (\cx+\s/2, \cy-\s/2) -- 
    (\cx+\s/2, \cy+\s/2) -- (\cx-\s/2, \cy+\s/2) -- cycle;

    \def\rows{5}
    \def\cols{5}

    \def\spacing{ \s/(\cols-1) }

    \draw[->, thick, line width=0.5mm, >=latex] (3.9,0.4) -- (0,0);

    \foreach \j in {0,1,2,3,4} {
        \ifnum\j=1
            \foreach \i in {0,1,2,3} {
                \pgfmathsetmacro{\x}{\cx - \s/2 + \i * \spacing + \spacing/2} 
                \pgfmathsetmacro{\y}{\cy - \s/2 + \j * \spacing}
                \fill (\x, \y) circle (2pt);  
            }
        \else
        \ifnum\j=3
            \foreach \i in {0,1,2,3} {
                \pgfmathsetmacro{\x}{\cx - \s/2 + \i * \spacing + \spacing/2} 
                \pgfmathsetmacro{\y}{\cy - \s/2 + \j * \spacing}
                \fill (\x, \y) circle (2pt);  
            }
        \else
            \foreach \i in {0,1,2,3,4} {
                \pgfmathsetmacro{\x}{\cx - \s/2 + \i * \spacing}
                \pgfmathsetmacro{\y}{\cy - \s/2 + \j * \spacing}
                \fill (\x, \y) circle (2pt);  
            }
        \fi
        \fi
    }

    \end{tikzpicture}}\caption{The construction in Proposition~\ref{thm:cons}, when $m=7$.
      Dotted edges represent the second largest distance $\Delta_2$.}
    \label{fig:construction1}

\end{figure}
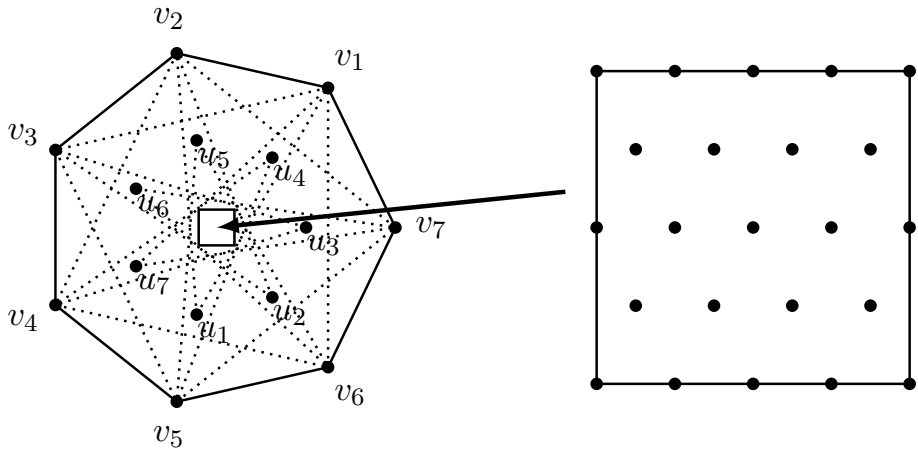

\section{Many large distance multiplicities among planar points}
\label{sec:many}

\subsection{Proof of Theorem~\ref{thm:many}}

For $n\in\NN$, let $[n]$ denote the set $\{1,2\ldots,n\}$. We first prove the following lemma by adapting
a well-known argument for counting representations of a natural number as the sum of two squares;
see, \eg, \cite[Chap.~3]{PA95} and~\cite[Chap.~2]{Gr85}.

\begin{lemma} \label{lem:many}
Let $r(n)$ denote the number of distinct ways in which $n \in \NN$ can be represented as
the sum of two squares. Then there exists a constant $c>0$ such that for infinitely many $n\in\NN$,
at least $n^{c/\log \log{n}}$ distinct elements $n' \in [n]$ have
\[ r(n') \geq n^{c/\log \log{n}}. \]
\end{lemma}
\begin{proof}
  Let $n=p_1 p_2 \cdots p_k$, where $p_j$ is the $j^{\mathrm{th}}$ smallest prime of the form $4m+1$.
  Since $p_k$ satisfies
  \[ c_1 k \log{k} \leq p_k \leq c_2 k \log{k}, \]
  for suitable constants $c_1,c_2>0$, this implies $k \geq  2c \log{n} / \log \log{n}$ for
  a suitable constant $c>0$. 
It is well-known that any such prime can be represented (uniquely) as the sum of two squares,
\ie,
\[ p_j = a_j^2 + b_j^2 =(a_j + b_j \I)(a_j-b_j \I), \]
where $\I=\sqrt{-1}$. 
There are $2^k$ subsets of $K=\{1,2,\ldots,k\}$, and out of these, exactly $2^{k-1}$
subsets have cardinality at least $k/2$. 

Fix any subset $K' \subseteq K$ of cardinality $|K'| \geq k/2$.
Let $n'= \prod_{j \in K'} p_j$. For each subset $J \subseteq K'$, 
\begin{align*}
\prod_{j \in J} (a_j + b_j \I) \prod_{j \in K' \setminus J} (a_j - b_j \I) &= A_J + B_J \I, \\ 
\prod_{j \in J} (a_j - b_j \I) \prod_{j \in K' \setminus J} (a_j + b_j \I) &= A_J - B_J \I,
\end{align*}
where $A_J$ and $B_J$ satisfy
\[ A_J^2 + B_J^2 = (A_J + B_J \I)(A_J - B_J \I)= \prod_{j \in K'} p_j =n' \leq n. \]
By the unique factorization theorem for complex integers, $A_J + B_J \I$ is different for
different choices of $J$, so we obtain
\[ r(n') \geq 2^{k/2} \geq n^{c/\log \log{n}}.\]

Since there are  $2^{k-1} \geq n^{c/\log \log{n}}$ distinct values $n' \in [n]$
that have been considered, the lemma is implied.
\end{proof}

\begin{proof}[Proof of Theorem~\ref{thm:many}]
  The proof follows a (now standard) argument of Erd\H{o}s~\cite{Er46} using
  the estimate in Lemma~\ref{lem:many}. Let $n_0\leq n$ be the largest integer
  such that $n_0=p_1p_2\dots p_k$, where $p_j$ is the $j^{\text{th}}$
  smallest prime of the form $4m+1$. Since
  $k=\Theta(\log{n} / \log\log{n})$, we
  have \[p_{k+1}=\Theta((k+1)\log(k+1))=\Theta(\log{n}),\] 
  namely, $n_{0}=\Omega(n/\log{n})$. By Lemma~\ref{lem:many} there
  exist $n_{0}^{\Omega(1/\log\log{n_0})}=n^{\Omega(1/\log\log{n})}$
  different values of $n'\in[n_0]$ that can be represented as the sum
  of two squares in $n^{\Omega(1/\log\log{n})}$ ways. For every such
  value of $n'$ there are $\Omega(n)$ points in the
  $\sqrt{n}\times\sqrt{n}$ grid, each of which has
  $n^{\Omega(1/\log\log{n})}$ neighbors at distance $\sqrt{n'}$. This
  completes the proof. 
\end{proof}

\subsection{Proof of Proposition~\ref{thm:m/9}}

\begin{proof}
  We prove the second estimate; the proofs of the other two estimates are analogous.
  Let $X$ be a $\sqrt{n} \times \sqrt{n}$ section of the integer grid, where $n=16k^2$.
  Then $X$ determines
  \[ m=(1 \pm o(1)) \frac{cn}{\sqrt{\log{n}}} \]
  distinct distances, for some $c>0$; see~\cite{Er46} or~\cite[Chap.~12]{PA95}.
  $X$ consists of $16$ smaller $k \times k$ sections $Y$,  each determining 
  \[ (1 \pm o(1)) \frac{cn}{16\sqrt{\log{(n/16)}}} =(1 \pm o(1)) \frac{m}{16} \]
  distinct distances. See Figure~\ref{fig:grid}. 

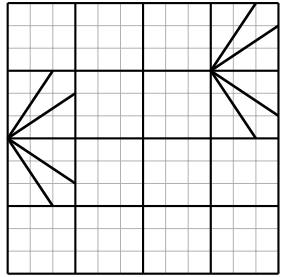
\begin{figure}[htbp]
    \centering
\begin{tikzpicture}[scale=0.3]

  \draw[step=1cm,very thin,gray!60] (0,0) grid (12,12);

  \foreach \x in {0,3,...,12} {
    \draw[thick] (\x,0) -- (\x,12);
  }
  \foreach \y in {0,3,...,12} {
    \draw[thick] (0,\y) -- (12,\y);
  }

  \draw[line width=1pt] (0,6) -- (3,8);
  \draw[line width=1pt] (0,6) -- (2,9);
  \draw[line width=1pt] (0,6) -- (3,4);
  \draw[line width=1pt] (0,6) -- (2,3);
  \draw[line width=1pt] (9,9) -- (12,11);
  \draw[line width=1pt] (9,9) -- (11,12);
  \draw[line width=1pt] (9,9) -- (12,7);
  \draw[line width=1pt] (9,9) -- (11,6);

\end{tikzpicture}
\caption{Multiplicities of distances in the grid.}
\label{fig:grid}
\end{figure}

Take any distance determined by a non-vertical and non-horizontal
segment $s=ab$ in $Y$, where $a$ and $b$ are its left and 
right endpoint, respectively. Observe that $s$ occurs at least
\[ 2k \cdot 3k \cdot4 + 2k \cdot 3k \cdot 2 = 36k^2 = \frac{9}{4} n \]
times in $X$. Indeed, the left degree of every point in the $6$ central left smaller sections 
is at least $4$ whereas the left degree of every point in the remaining $6$ left smaller sections 
is at least $2$. Note that the number of distances determined by a
vertical or horizontal segment in $Y$ is at most $k=o(m)$. This
justifies the second estimate.  

The first and the third estimate are obtained analogously by subdividing $X$ into $9$
and $25$ smaller sections, respectively.
\end{proof}

\section{On the differences $a_k(X) - a_{k+1}(X)$} \label{sec:diff}

Using an inductive construction, Erd\H{o}s and Purdy~\cite{EP76}
showed that the maximum number of times the unit distance occurs among $n$ points
in the plane, no three of which are collinear, is at least
$\Omega(n\log{n})$. Our proof of Theorem~\ref{thm:diff} can be viewed
as a refinement of their argument. 

\begin{proof}[Proof of Theorem~\ref{thm:diff}]
  Let $d_1,\ldots,d_k>0$ be arbitrary pairwise distinct distances. Let
  $X_0(m)$ denote a configuration of $m$ points and let $X_i(2m)$
  denote the union of $X_{i-1}(m)$ and a translate 
  of $X_{i-1}(m)$ by distance $d_i$ in some generic direction, so that
  none of the segments connecting the two copies duplicates a distance
  other than $d_i$. This is feasible since it amounts to excluding a
  set of directions of measure zero. 
  We start from a single point and apply translates by
  $d_1,\ldots,d_k$ in a cyclic fashion. The resulting set after $k$
  steps has $2^k \leq n$ points.  

  For any $1\leq i \leq k$, the multiplicity $T_i(n)$ of $d_i$ in a set of $n$ points constructed in the above way
  satisfies the recurrence
  \[ T_i(n) = 2^k T_i\left( n/2^k \right) + n/2, \ \ T_i(1)=0. \]
  Its solution satisfies
  \[ T_i(n) \geq \frac{n}{2} \log_{2^k}{n} = \frac{n}{2k} \log{n}. \] 
  In the inductive step corresponding to $d_i$, we ensure that none of the segments connecting the two copies
  duplicate a distance other than $d_i$. Consequently, any distance other than $d_1,\ldots,d_k$ occurs at most
  $n$ times. In conclusion, we have 
\[ a_1(X),\ldots,a_k(X) = \Omega\left(\frac{n}{k} \log{n}\right),
\text{ and } a_{k+1}(X), a_{k+2}(X), \ldots \leq n, \]
as required. 
\end{proof}

\section{Point sets with distinct distance multiplicities}\label{sec:distinct}

Given a set $X\subseteq\RR^2$ of $n\geq2$ points, which contains
$m=m(X)\leq\binom{n}{2}$ distinct distances, recall that
$a(X)=(a_1(X),\dots,a_m(X))$ consists of the multiplicities of all distances
ordered by $a_1(X) \geq a_2(X) \geq \dots \geq a_m(X)$.
How many distinct values can $a(X)$ contain? At most $n-1$, this
follows easily from $\sum_{k=1}^m a_k(X)=\binom{n}{2}$. Moreover, when
$a(X)$ contains $n-1$ distinct values, then
$a(X)=(n-1,n-2,\dots,1)$. One can observe that if $X$ consists of
equidistant points on 
a line or on a circle, see Figure~\ref{fig:circular_arc}~\textup{(a)} and~\textup{(b)} for an
illustration, then $a(X)=(n-1,n-2,\dots,1)$. Are there other 
constructions of $X$ that achieve $a(X)=(n-1,n-2,\dots,1)$?
Erd\H{o}s~\cite[p.~135]{Er84} conjectured the answer to be negative, when $n$ is large. Here we give
a simple counterexample to this conjecture. 

\begin{observation} \label{obs:simple}
  Let $\gamma$ be a circular arc subtending a center angle $<\pi/3$ on the circle $C$ of unit radius
  centered at $c$. Let $X$ consist of $c$ together with a set of $n-1$ equidistant points on $\gamma$.
  Then $a(X)=(n-1,n-2,\ldots,1)$. See Figure~\ref{fig:circular_arc}~\textup{(c)} for an illustration.
\end{observation}

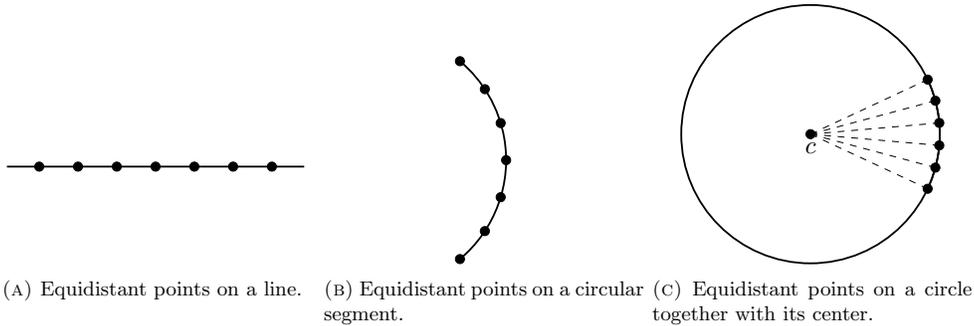
\begin{figure}[h]
  \centering
    \resizebox{0.8\textwidth}{!}{ %
    \begin{subfigure}{0.3\textwidth}
        \centering
        \begin{tikzpicture}
          \def\n{7}
            \def\spacing{0.6} 
            \def\shift{1.5}
            \foreach \i in {0,...,6} {
                \filldraw[black] (\i*\spacing,\shift) circle (2pt);
            }

            \draw[thick] (-0.5,\shift) -- (\n*\spacing - \spacing + 0.5,\shift);

            \draw (0.5,0) circle (0pt);
            
        \end{tikzpicture}

        \caption{Equidistant points on a line. \newline}
    \end{subfigure}
    \hfill
    \begin{subfigure}{0.3\textwidth}
        \centering
        \begin{tikzpicture}
            \def\radius{2}
            \def\angle{100} 
            \def\n{7} 

            \foreach \i in {0,...,6} {
                \pgfmathsetmacro{\theta}{- \angle/2 + \i * \angle / (\n-1)}
                \pgfmathsetmacro{\x}{\radius * cos(\theta)}
                \pgfmathsetmacro{\y}{\radius * sin(\theta)}
                
                \filldraw[black] (\x,\y) circle (2pt);
            }

            \pgfmathsetmacro{\startX}{\radius * cos(-\angle/2)}
            \pgfmathsetmacro{\startY}{\radius * sin(-\angle/2)}
            \draw[thick] (\startX, \startY) 
              arc[start angle=-\angle/2, end angle=\angle/2, x radius=\radius, y radius=\radius];
        \end{tikzpicture}
        \caption{Equidistant points on a circular segment.}
    \end{subfigure}
    \hfill
    \begin{subfigure}{0.3\textwidth}
        \centering
        \begin{tikzpicture}
            \def\radius{2}
            \def\angle{50} 
            \def\n{7}
            \def\n1{6}
            \def\center{(0,0)}

            \draw[thick] \center circle (\radius);

            \filldraw[black] \center circle (2pt) node[below] {$c$};

            \foreach \i in {1,...,\n1} {
                \pgfmathsetmacro{\theta}{- \angle/2 + (\i-1)*\angle/(\n1-1)}
                \pgfmathsetmacro{\x}{\radius * cos(\theta)}
                \pgfmathsetmacro{\y}{\radius * sin(\theta)}
                
                \filldraw[black] (\x,\y) circle (2pt);

                \draw[dashed] \center -- (\x,\y);
            }

            \pgfmathsetmacro{\startX}{\radius * cos(-\angle/2)}
            \pgfmathsetmacro{\startY}{\radius * sin(-\angle/2)}
            \draw[thick] (\startX, \startY) 
              arc[start angle=-\angle/2, end angle=\angle/2, x radius=\radius, y radius=\radius];

        \end{tikzpicture}
        \caption{Equidistant points on a circle together with its center.}
    \end{subfigure}
    }

    \caption{Three configurations $X$ satisfying $a(X)=(n-1,n-2,\ldots,1)$ for $n=7$.}
    \label{fig:circular_arc}
\end{figure}

\begin{proof}
  The multiplicities of the $n-1$ points on $\gamma$ are $1,2,\ldots,n-2$.
  Since the multiplicity of the unit distance is $n-1$, we have $a(X)=(n-1,n-2,\ldots,1)$.
  Finally, $X$ is not contained in any line or circle.
\end{proof}

One may further ask:

\begin{problem}
For sufficiently large $n\in\mathbb{N}$, are the examples in
Figure~\ref{fig:circular_arc} the only
point sets with $a(X)=(n-1,n-2,\ldots,1)$? Are these the only ones
with pairwise distinct distance multiplicities? 
\end{problem}

We answer the latter question in the negative. We also show that an integer grid is not a valid candidate.

\begin{proposition} \label{prop:distinct-mu}
For every $n\in\NN$, there is a set $X\subseteq\RR^2$ of $n$ points with
pairwise distinct distance multiplicities and $a(X) \neq (n-1,n-2,\ldots,1)$. 
\end{proposition}
\begin{proof}
We present the proof for odd $n$; the case of even $n$ is analogous and left to the reader.
Let $X$ be a piece of the hexagonal lattice of side length $1$ with $n=2k+1$
points placed on two adjacent horizontal lines $\ell_1,\ell_2$, so that 
$|X \cap \ell_1|=k+1$ and $|X \cap \ell_2|=k$. See Figure~\ref{fig:tworows}~(left) for an illustration.

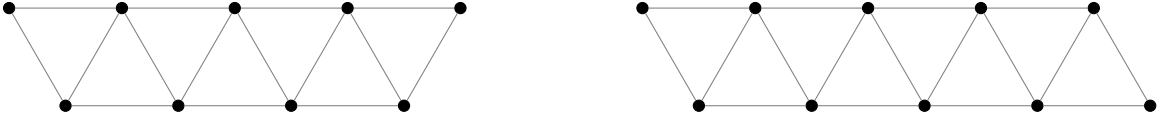
\begin{figure}[htbp]
\centering
\begin{subfigure}{0.47\textwidth}
\centering
\begin{tikzpicture}[
  every node/.style={circle, draw, fill=black, inner sep=1.5pt},
]

\def\L{1.5} 
\def\h{1.5*sqrt(3)/2} 

\foreach \i in {0,...,8} {
  \pgfmathsetmacro{\x}{\i*\L/2}
  \pgfmathsetmacro{\y}{mod(\i,2)==0 ? 0 : -\h}
  \node (p\i) at (\x,\y) {};
}

\foreach \i in {0,...,7} {
  \pgfmathtruncatemacro{\j}{\i+1}
  \draw[gray] (p\i) -- (p\j);
}
\foreach \i in {0,...,6} {
  \pgfmathtruncatemacro{\j}{\i+2}
  \draw[gray] (p\i) -- (p\j);
}

\end{tikzpicture}
\end{subfigure}
\hfill
\begin{subfigure}{0.47\textwidth}
\centering
\begin{tikzpicture}[
  every node/.style={circle, draw, fill=black, inner sep=1.5pt},
]

\def\L{1.5} 
\def\h{1.5*sqrt(3)/2} 

\foreach \i in {0,...,9} {
  \pgfmathsetmacro{\x}{\i*\L/2}
  \pgfmathsetmacro{\y}{mod(\i,2)==0 ? 0 : -\h}
  \node (p\i) at (\x,\y) {};
}

\foreach \i in {0,...,8} {
  \pgfmathtruncatemacro{\j}{\i+1}
  \draw[gray] (p\i) -- (p\j);
}
\foreach \i in {0,...,7} {
  \pgfmathtruncatemacro{\j}{\i+2}
  \draw[gray] (p\i) -- (p\j);
}

\end{tikzpicture}
\end{subfigure}
\caption{Point sets ($n=9$ and $n=10$) with pairwise distinct distance multiplicities
  and $a(X) \neq (n-1,n-2,\ldots,1)$.}
\label{fig:tworows}
\end{figure}

There are two types of distances in $X$, integer and irrational. The integer
distances are $\{1,\ldots,k\}$, determined by points on the same horizontal
line, or by points with consecutive $x$-coordinates on different horizontal
lines. The irrational distances occur between nonconsecutive points on different
horizontal lines. 
Let these be $d_1<d_2< \cdots < d_{k-1}$, where $d_j = \sqrt{j^2+j+1}$ for $j=1,\ldots,k-1$.
It is not difficult to verify that
\begin{enumerate} [(i)] \itemsep 2pt
\item $\mu(X,1)=4k-1$.
\item $\mu(X,j) = 2(k-j)+1, \text{ for } j=2,\ldots,k$.
\item $\mu(X,d_j) = 2(k-j), \text{ for } j=1,\ldots,k-1$.
\end{enumerate}
The multiplicities are clearly distinct and this completes the proof.
\end{proof}

\begin{observation}
    Let $k\geq 4$. In the $k \times k$ grid there are two distances which appear exactly $8$ times each.
\end{observation}
\begin{proof}
The distance $d_1=\sqrt{(k-1)^2+(k-2)^2}$ appears among the pairs: 
\begin{gather*}\{(0,0),(k-1,k-2)\},  \{(0,0),(k-2,k-1)\}, \{(1,0),(k-1,k-1)\}, 
\{(0,1),(k-1,k-1)\}, \\
\{(k-1,0),(0,k-2)\}, \{(k-1,0),(1,k-1)\}, \{(k-2,0),(0,k-1)\}, \{(k-1,1),(0,k-1)\}.
\end{gather*}
The distance $d_1$ can only appear between point pairs $(x_1,y_1)$ and $(x_2,y_2)$ where
$(|x_1-x_2|,|y_1-y_2|)\in\{(k-1,k-2),(k-2,k-1)\}$, and thus cannot appear more 
than eight times, as shown above. The distance $d_2=\sqrt{2(k-2)^2}$ appears among the pairs: 
\begin{gather*}\{(0,0),(k-2,k-2)\}, \{(1,0),(k-1,k-2)\}, \{(0,1),(k-2,k-1)\}, 
\{(1,1),(k-2,k-2)\}, \\
\{(k-1,0),(1,k-2)\}, \{(k-1,1),(1,k-1)\}, \{(k-2,0),(0,k-2)\}, \{(k-2,1),(0,k-1)\}.
\end{gather*}
Note that this is an exhaustive list of all point pairs $(|x_1-x_2|,|y_1-y_2|)=(k-2,k-2).$
It is not possible that $|x_1-x_2|=k-1$ (respectively $|y_1-y_2|=k-1$), since the equation
$$ (k-1)^2+(x)^2=2(k-2)^2 $$ 
does not have a solution. Indeed, if

$$ x^2= 2(k-2)^2-(k-1)^2=k^2-6k+7,$$
then
$$ k=3 \pm \sqrt{3^2-(7-x^2)}=3\pm \sqrt{2+x^2}, $$
but $x^2+2$ is not square for $x\in \mathbb{N}$.
\end{proof}

\end{document}